\documentclass[a4paper,12pt]{amsart} 

\usepackage{amssymb,amsmath}

\usepackage[right=2cm, left=2cm, top=2cm, bottom=2cm]{geometry}

\newtheorem{thm}{Theorem}[section] 
\newtheorem*{thm*}{Theorem}
\newtheorem{prop}[thm]{Proposition}
\newtheorem{cor}[thm]{Corollary} 
\newtheorem{lem}[thm]{Lemma}

\theoremstyle{definition} 
 
\newtheorem{rem}[thm]{Remark}  
\newtheorem{exa}[thm]{Example}

\numberwithin{equation}{section}

\newcommand{\skal}[2]{\langle #1,#2\rangle}
\newcommand{\alg}[1]{\mathfrak{#1}}

\begin{document}

\title{Two notes on $Spin(7)$--structures}
\author{Kamil Niedzia\l omski}
\date{}

 
\address{
Department of Mathematics and Computer Science \endgraf
University of \L\'{o}d\'{z} \endgraf
ul. Banacha 22, 90-238 \L\'{o}d\'{z} \endgraf
Poland
}
\email{kamil.niedzialomski@wmii.uni.lodz.pl}

\begin{abstract}
We derive the explicit formula for the intrinsic torsion of a ${\rm Spin}(7)$-structure on a $8$--dimensional Riemannian manifold $M$. Here, the intrinsic torsion is a difference of the minimal ${\rm Spin}(7)$--connection and the Levi-Civita connection. Hence it is a a section of a bundle $T^{\ast}M\otimes\alg{spin}^{\bot}(M)$. The formula relates the intrinsic torsion with the Lee form $\theta$ and $\Lambda^3_{48}$--component $(\delta\Phi)_{48}$ of a codifferential $\delta\Phi$ of the $4$--form defining a given structure. Using the formula obtained, we compute the condition for a ${\rm Spin}(7)$ structure of type $\mathcal{W}_8$ to be (second order) nearly parallel. Moreover, applying the divergence formula obtained by the author for general Riemannian $G$--structure in another article, we rediscover the well known formula for the scalar curvature in terms of norms of $\theta$, $(\delta\Phi)_{48}$ and the divergence ${\rm div}\theta$. We justify the formula on appropriate examples.   
\end{abstract}

\maketitle

\section{Introduction}

Among the possible Riemannian holonomy groups of simply connected, irreducible and non--symmetric spaces there are two called exceptional \cite{MB}. One of them is ${\rm Spin}(7)$--structure on a $8$--dimensional Riemannian manifold. Such structures have been extensively studied by many authors in many contexts. The first approach to classifying such structures was taken by M. Fernandez \cite{MF}. She showed that the covariant derivative $\nabla\Phi$ of the $4$--form $\Phi$ defining a ${\rm Spin}(7)$--structure, determines all possible classes of ${\rm Spin}(7)$--structure. Moreover, the holonomy group $SO(n)$ reduces to ${\rm Spin}(7)$ precisely when $\nabla\Phi$ vanishes. The space $\mathcal{W}$ of elements of $T^{\ast}M\otimes \Lambda^4(M)$ with the "symmetries" as covariant derivative $\nabla\Phi$ splits into two irreducible $\mathcal{W}_8$ and $\mathcal{W}_{48}$--modules. ${\rm Spin}(7)$. Moreover, M. Fernandez showed that instead of $\nabla\Phi$, we may consider $d\Phi$. This implies that it is also sufficient to consider the codifferential $\delta\Phi$. Knowing the decomposition of $\Lambda^3(M)$ into two irreducible modules, $\Lambda^3(M)=\Lambda^3_8\oplus\Lambda^3_{48}$, it follows that the ${\rm Spin}(7)$--structure is determined by the components $(\delta\Phi)_8$ and $(\delta\Phi)_{48}$, where the first depends on the Lee form $\theta=\star(\delta\Phi\wedge\Phi)$. 

The intrinsic torsion $\xi$ of a ${\rm Spin}(7)$--structure is a $(1,2)$--tensor measuring the defect of the Levi--Civita connection $\nabla$ to be a ${\rm Spin}(7)$--connection. Analogously, $\xi$ determines all classes of ${\rm Spin}(7)$--structures, since $\nabla_X\Phi=\xi_X\Phi$. The intrinsic torsion has been considered by many authors. However, the explicit formula, up to the author's knowledge, has been obtained by F. M. Cabrera and S. Ivanov in \cite{CI}, S. Karigiannis in \cite{SK2}, and M. Merchan in \cite{MM}. The formulas obtained by these authors rely on $\nabla\Phi$ or $d\Phi$ and $\theta$ rather on $\delta\Phi$. We propose a formula, using different approach, depending on $\delta\Phi$ and $\theta$. The approach in this article is also justified by the fact that the main emphasis is put on
\begin{itemize}
    \item deriving a condition for a ${\rm Spin}(7)$--structure to be second order nearly parallel, which involves computation of the covariant derivative of the intrinsic torsion,
    \item deriving a divergence formula relating the scalar curvature $s$, the Lee form and the codifferential $\delta\Phi$. Such divergence formula has been obtained, for example in \cite{SI}, but using different methods.
\end{itemize}
Before stating our main results, let us introduce the necessary objects. The intrinsic torsion $\xi$ allows one to define four interesting objects -- symmetric $2$--tensor $r$, alternation $\xi^{\rm alt}_XY$, symmetrization $\xi^{\rm sym}_XY$ and so-called characteristic vector field $\chi$ -- in the following way:   
\begin{align*}
& r(X,Y) = \sum_i g(\xi_Xe_i, \xi_Ye_i), &&\chi = \sum_i \xi_{e_i}e_i,\\
& \xi^{\rm alt}_XY = \frac{1}{2}(\xi_XY-\xi_YX),  &&\xi^{\rm sym}_XY = \frac{1}{2}(\xi_XY+\xi_YX).
\end{align*}
Here, $(e_i)$ is an orthonormal basis.

The first part of this note is to look for ${\rm Spin}(7)$--structures of type $\mathcal{W}_8$ (locally conformal parallel) such that the intrinsic torsion $\xi$ is parallel with respect to the connection $\nabla^{{\rm Spin}(7)} = \nabla + \xi$ and such that $r$ is a parallel tensor with respect to the same connection $\nabla^{{\rm Spin}(7)}$. In this case, we obtain the following equivalences:
\begin{equation*}
\nabla^{{\rm Spin}(7)}\xi=0 \quad\Longleftrightarrow\quad \nabla^{{\rm Spin}(7)}r=0 \quad\Longleftrightarrow\quad \nabla^{{\rm Spin}(7)}\theta = 0.
\end{equation*}
They imply non--existence of locally conformal parallel ${\rm Spin}(7)$--structures with parallel intronsic torsion or parallel tensor $r$.

The second part is devoted to a general formula relating the last three objects with the curvatures of considered $G$--structure, in our case ${\rm Spin}(7)$--structure \cite{KN}:
\begin{equation*}
{\rm div}\,\chi=\tilde{s}+|\chi|^2+|\xi^{\rm alt}|^2-|\xi^{\rm sym}|^2,
\end{equation*}
where $\tilde{s}$ is a type of scalar curvature obtained by a decomposition of the Riemann curvature tensor. It is interesting that in the case of ${\rm Spin}(7)$ we have $\chi=-\frac{7}{16}\theta^{\sharp}$, $\tilde{s}=\frac{1}{4}s$, and $|\xi^{\rm alt}|^2-|\xi^{\rm sym}|^2$ is a combination of $|\theta|^2$ and $|(\delta\Phi)_{48}|^2$. This allows us to prove the main formula:
\begin{equation}\label{eq:main}
s=\frac{21}{8}|\theta|^2-\frac{1}{2}|(\delta\Phi)_{48}|^2+\frac{7}{2}{\rm div}\,\theta.
\end{equation}
Such formula have been already derived by S. Ivanov \cite{SI} and, later, by U. Fowder \cite{UF}. However, the formula for $s$ in \cite{SI} depends on the norm $\|T^c\|$ of the characteristic torsion $T^c$ and is obtained with the use of the Lichnerowicz formula relating the Laplace operator and the Dirac operator, while the approach in \cite{UF} is based on representation theory. Our approach to this formula in different. In the case of balanced structure, the formula becomes even simpler, since the Lee form vanishes,
\begin{equation*}
s=-\frac{1}{2}|(\delta\Phi)_{48}|^2.
\end{equation*}
This gives constraints on the geometry of such structures -- the scalar curvature is non--positive. This was already noticed in \cite{SI}. We also compare the main divergence formula with the results obtained in \cite{BHL}, where the authors use a Bochner--type integral formula to derive different formulas for Riemannian $G$--structures and, in particular, for ${\rm Spin}(7)$--structures. In fact, integrating \eqref{eq:main} and assuming $M$ is closed, we get
\begin{equation*}
\int_M s\,{\rm vol}_M=\int_M(\frac{21}{8}|\theta|^2-\frac{1}{2}|(\delta\Phi)_{48}|^2)\,{\rm vol}_M.
\end{equation*} 
We end the article with appropriate examples. We justify the formula \eqref{eq:main} by computing both sides in two cases: $M(k)\times\mathbb{T}^5$ and $M(k)\times H/\Gamma\times\mathbb{T}^2$, where $M(k)$ is a certain quotient of the three--dimensional Lie group, $H$ is the Heisenberg group and $\mathbb{T}$ is a torus.

\section{Intrinsic torsion of a ${\rm Spin}(7)$--structure}

\subsection{Algebraic background}

Let us briefly introduce the algebraic background concerning the Lie group ${\rm Spin}(7)$, its Lie algebra $\alg{spin}(7)$ and irreducible representations of ${\rm Spin}(7)$ acting on the spaces of skew forms.  

Consider the Euclidean space $\mathbb{R}^8$ with the canonical orthonormal basis $\{e_0,e_1,\ldots,e_7\}$ and a $4$--form $\Phi$ given by
\begin{multline}\label{eq:Phi}
\Phi=e^{0123}-e^{0145}-e^{0167}-e^{0246}+e^{0257}-e^{0347}-e^{0356}\\
+e^{4567}-e^{2367}-e^{2345}-e^{1357}+e^{1346}-e^{1256}-e^{1247}.
\end{multline}
Notice that we may write $\Phi$ in the form $\Phi=e^0\wedge\varphi+\star(e^0\wedge\varphi)$, where $\varphi$ defines the Lie group $G_2$ and $\star$ is the Hodge star operator. Notice that $\Phi$ is self-dual, $\star \Phi=\Phi$. The group ${\rm Spin}(7)$ is the group of elements $g\in SO(8)$ leaving $\Phi$ invariant,
\begin{equation*}
{\rm Spin}(7)=\{g\in SO(8)\mid g^{\ast}\Phi=\Phi\}.
\end{equation*} 
With this approach, we get the action of ${\rm Spin}(7)$ on $\mathbb{R}^8$ and thus on all tensors and differential forms. We will frequently use the following relation \cite{SK}
\begin{equation*}
\star(\Phi\wedge \eta)=\,\eta\lrcorner\Phi,
\end{equation*}
valid for any $k$--form $\eta$.

There are the following decompositions into irreducible ${\rm Spin}(7)$--modules \cite{SW,MF, SK}
\begin{align*}
\Lambda^2(\mathbb{R}^8) &=\Lambda^2_7\oplus\Lambda^2_{21},\\
\Lambda^3(\mathbb{R}^8) &=\Lambda^3_8\oplus\Lambda^3_{48},\\
\Lambda^5(\mathbb{R}^8) &=\Lambda^5_8\oplus\Lambda^3_{48},
\end{align*}
where the lower index indicates the dimension and where each submodule can be described as 
\begin{align*}
\Lambda^2_7 &=\{\omega\mid \star(\Phi\wedge\omega)=3\omega\}
=\{\omega+\omega\lrcorner\Phi\mid \omega\in \Lambda^2((\mathbb{R}^8)^{\ast})\} \\
&=\{\omega_{\xi}\mid \xi\in\alg{spin}(7)^{\bot}\},\\
\Lambda^2_{21} &=\{\omega\mid \star(\Phi\wedge\omega)=-\omega\}
=\{\omega_{\xi}\mid \xi\in\alg{spin}(7)\},\\
\Lambda^3_8 &=\{\star(\Phi\wedge\eta)\mid \eta\in\Lambda^1((\mathbb{R}^8)^{\ast}\},\\
\Lambda^3_{48} &=\{\gamma\mid \Phi\wedge\gamma=0\},\\
\Lambda^5_8 &=\{\Phi\wedge\eta\mid \eta\in\Lambda^1(\mathbb{R}^8)^{\ast}\},\\
\Lambda^5_{48} &=\{\mu\mid \Phi\wedge(\star\mu)=0\}.
\end{align*}
In the above definitions, $\omega_{\xi}$ stands for a $2$--form induced by a skew--matrix $\xi\in\alg{so}(8)$. The spaces $\Lambda^k$ for$k=5,6,7$ can be described analogously by the fact that the Hodge operator $\star:\Lambda^{8-k}(\mathbb{R}^8)\to \Lambda^k(\mathbb{R}^8)$ is an ${\rm Spin}(7)$-isomorphism. Moreover, we have the following isomorphisms
\begin{equation*}
\Lambda^1\ni \alpha\mapsto \Phi\wedge\alpha\in\Lambda^5_8,\quad
\Lambda^1\ni \alpha\mapsto \star(\Phi\wedge\alpha)\in\Lambda^3_8
\end{equation*}
which follow by the following identity \cite{SK}
\begin{equation}\label{eq:isomorphism}
\star(\Phi\wedge\star(\Phi\wedge\eta))=-7\eta,\quad \eta\in\Lambda^1.
\end{equation}
Moreover, if $\mu=\Phi\wedge\eta\in\Lambda^5_8$, alternatively, $\gamma=\star(\Phi\wedge\eta)\in\Lambda^3_8$, then a one--form $\eta$ may be recovered as follows (using \eqref{eq:isomorphism})
\begin{equation*}
\eta=-\frac{1}{7}\star(\Phi\wedge\star\mu),\quad\textrm{alternatively},\quad \eta=-\frac{1}{7}\star(\Phi\wedge\gamma).
\end{equation*}

For any $2$--form $\eta$ denote its components with respect to the decomposition $\Lambda^2(\mathbb{R}^8)=\Lambda^2_7\oplus\Lambda^2_{21}$ by $\eta_7$ and $\eta_{21}$, respectively. Analogously, we decompose $3$--form $\gamma$ as $\gamma=\gamma_8+\gamma_{48}$ and $5$--form $\mu$ as $\mu=\mu_8+\mu_{48}$.

We will need the following relations taken from \cite{SK} (Proposition 4.3.1 and Appendix A):
\begin{align}
\star(v\lrcorner\alpha) &=(-1)^{k+1}v^{\flat}\wedge\star\alpha,\label{eq:A0}\\
v\lrcorner\star\alpha &=(-1)^k\star(v^{\flat}\wedge\alpha),\label{eq:A1}\\
v\lrcorner w\lrcorner\Phi &=-3(v^{\flat}\wedge w^{\flat})_7+(v^{\flat}\wedge w^{\flat})_{21},\label{eq:A2}\\
\star((v\lrcorner\Phi)\wedge(w\lrcorner\Phi)) &=2(v^{\flat}\wedge w^{\flat})_7-6(v^{\flat}\wedge w^{\flat})_{21},\label{eq:A3}\\
(v\lrcorner w\lrcorner\Phi)\wedge\Phi &=-(v\lrcorner\Phi)\wedge(w\lrcorner\Phi)-7\star(v^{\flat}\wedge w^{\flat}),
\label{eq:A4}\\
\Phi\wedge(\gamma\lrcorner\Phi) &=7\star\gamma_8\label{eq:A5},
\end{align}
where $\alpha$ is any $k$--form and $\gamma$ is a $3$--form. Relations \eqref{eq:A3} and \eqref{eq:A4} imply the following formula
\begin{equation}\label{eq:A6}
(\eta\lrcorner\Phi)\wedge\Phi=9\star\eta_7+\star\eta_{21}
\end{equation}
for any $2$--form $\eta$.

Notice that, by the description of $\Lambda^2_7$ and $\Lambda^2_{21}$ we see that
\begin{align*}
\omega_7 &=\frac{1}{4}(\omega+\omega\lrcorner\Phi)=\frac{1}{4}(\omega+\star(\Phi\wedge\omega)),\\
\omega_{21} &=\frac{1}{4}(3\omega-\omega\lrcorner\Phi)=\frac{1}{4}(3\omega-\star(\Phi\wedge\omega)).
\end{align*}

\subsection{Explicit formula for intrinsic torsion of a ${\rm Spin}(7)$--structure}

Let $M$ be a ${\rm Spin}(7)$--structure, i.e. $M$ is a Riemannian $8$--dimensional manifold such that there is a reduction $P$ of the orthonormal frame bundle ${\rm SO}(M)$ to the subgroup ${\rm Spin}(7)\subset SO(8)$. Alternatively, it is equivalent to the existence of a global $4$--form $\Phi$, which locally can be written in the form \eqref{eq:Phi}. Denote by $\nabla$ the Levi-Civita connection on $M$ and by $\omega$ the corresponding connection form on ${\rm SO}(M)$. The ${\rm Ad}$--invariant decomposition of the Lie algebra $\alg{so}(8)$ given by
\begin{equation}\label{eq:so8decomposition}
\alg{so}(8)=\alg{spin}(7)\oplus\alg{spin}(7)^{\bot},
\end{equation} 
induces decomposition of the connection form $\omega$: $\omega=\omega_{\alg{spin}(7)}+\omega_{\alg{spin}(7)^{\bot}}$. The first component induces the Riemannian connection $\nabla^{{\rm Spin}(7)}$ on $M$, whereas the second component induces (with the minus sign) $(1,2)$--tensor $\xi$ called the intrinsic torsion. In other words,
\begin{equation*}
\nabla^{{\rm Spin}(7)}=\nabla+\xi.
\end{equation*} 
By definition, it follows that $\xi\in T^{\ast}M\otimes\alg{ spin}^{\bot}(M)$, where ${\rm spin}^{\bot}(M)=P\times_{{\rm Spin}(7)}\alg{spin}(7)^{\bot}$ is a subbundle of skew endomorphisms of $M$ or, equivalently, of $2$--forms. 

The space of all possible intrinsic torsions $\mathcal{T}=T^{\ast}M\otimes \alg{spin}^{\bot}(M)$ splits into two irreducible ${\rm Spin}(7)$--modules, $\mathcal{T}=\mathcal{W}_{48}\oplus\mathcal{W}_8$. They can be described as follows \cite{MF, SK}
\begin{align*}
&\textrm{balanced}\,(\mathcal{W}_{48}): &&\, \theta=0,\\ 
&\textrm{locally conformal parallel}\,(\mathcal{W}_8): &&\, d\Phi=\theta\wedge\Phi.
\end{align*}
where $\theta$ is the Lee form of a ${\rm Spin}(7)$--structure,
\begin{equation*}
\theta=\frac{1}{7}\star(\delta\Phi\wedge\Phi).
\end{equation*}
Let us elaborate on that. For a $\mathcal{W}_{48}$--structure, we get $0=\Phi\wedge\delta\Phi$, which implies $\delta\Phi\in\Lambda^3_{48}$. Equivalently, $d\Phi\in \Lambda^5_{48}$. For a $\mathcal{W}_8$--structure, by definition, we have $d\Phi\in \Lambda^5_8$ and $\delta\Phi=-\star(\Phi\wedge\theta)\in\Lambda^3_8$.

The following useful result is well known.

\begin{lem}\label{lem:Tdecmposition}
The following relations hold
\begin{equation*}
(\delta\Phi)_8=-\star(\Phi\wedge\theta)\quad\textrm{and}\quad |(\delta\Phi)_8|^2=7|\theta|^2.
\end{equation*}
\end{lem}
\begin{proof}
By \eqref{eq:A6} we have
\begin{equation*}
(\delta\Phi)_8=-\frac{1}{7}\star(\Phi\wedge((\delta\Phi)\lrcorner\Phi)=-\frac{1}{7}\star(\Phi\wedge\star(\delta\Phi\wedge\Phi))=-\star(\Phi\wedge\theta).
\end{equation*}
Thus, by the first part and identity \eqref{eq:isomorphism}, we get
\begin{equation*}
|(\delta\Phi)_8|^2{\rm vol}_M=|\Phi\wedge\theta|^2{\rm vol}_M=(\Phi\wedge\theta)\wedge\star(\Phi\wedge\theta)=7\theta\wedge\star\theta=7|\theta|^2{\rm vol}_M.
\end{equation*}
\end{proof}

Let us describe the intrinsic torsion $\xi$ explicitly. It has been already done at least in three different ways \cite{CI,SK2,MM}. Inspired by approach in \cite{SI} and \cite{IA, TF} we give another formula for the intrinsic torsion, most suitable for our purposes. Firstly, it can be shown \cite{TF, SI} that on any ${\rm Spin}(7)$--structure there is a unique characteristic connection $\nabla^c$ defined by the condition that it preserves $\Phi$, hence $g$ as well, and has a totally skew--symmetric torsion. This is a Riemannian connection, which is in fact ${\rm Spin}(7)$--connection and has totally skew--symmetric torsion $T^c$. The formula for $T^c$ is the following \cite{SI}
\begin{equation*}
T^c=-\delta\Phi-\frac{7}{6}\star(\theta\wedge\Phi).
\end{equation*}
Notice that by Lemma \ref{lem:Tdecmposition}, $T^c$ can be rewritten as
\begin{equation*}
T^c=\frac{1}{6}(\delta\Phi)_8-(\delta\Phi)_{48},
\end{equation*}
which implies
\begin{equation}\label{eq:Tc_norm}
|T^c|^2=\frac{1}{36}|(\delta\Phi)_8|^2+|(\delta\Phi)_{48}|^2
=\frac{7}{36}|\theta|^2+|(\delta\Phi)_{48}|^2.
\end{equation}
Moreover, it is known \cite{IA, TF, SI} that the characteristic connection exists for any ${\rm Spin}(7)$--structure and the intrinsic torsion is of the form
\begin{equation}\label{eq:inttorsioncharacteristictorsion}
\xi_X=\frac{1}{2}(X\lrcorner T^c)_7.
\end{equation}
This relation can be proven directly (see \cite{IA, TF} or \cite{MM} (proof of Proposition 5.1) using the spinorial approach).

\begin{prop}\label{prop:intrinsictorsion}
The intrinsic torsion of a ${\rm Spin}(7)$--structure can be described in the following ways 
\begin{equation*}
\xi_X =-\frac{1}{2}(X\lrcorner\delta\Phi)_7+\frac{7}{4}(X^{\flat}\wedge\theta)_7\quad\textrm{or}\quad \xi_X=-\frac{1}{2}(X\lrcorner(\delta\Phi)_{48})_7+\frac{1}{4}(X^{\flat}\wedge\theta)_7.
\end{equation*}
\end{prop}
\begin{proof}
By \eqref{eq:inttorsioncharacteristictorsion} using the formula for the projection $\pi_7$ and the characteristic torsion $T^c$, we have
\begin{align*}
2\xi_X &=\pi_7(X\lrcorner T^c) &\\
&=-(X\lrcorner\delta\Phi)_7-\frac{7}{24}\left(X\lrcorner\star(\theta\wedge\Phi)+\star((X\lrcorner\star(\theta\wedge\Phi))\wedge\Phi)\right).
\end{align*}
We will rewrite last two components. Firstly, by \eqref{eq:A1} and \eqref{eq:A2} we have
\begin{equation*}
X\lrcorner\star(\theta\wedge\Phi)=X\lrcorner\theta^{\sharp}\lrcorner\Phi=-3(X^{\flat}\wedge\theta)_7+(X^{\flat}\wedge\theta)_{21}.
\end{equation*} 
Further, by \eqref{eq:A6}
\begin{align*}
\star((X\lrcorner\star(\theta\wedge\Phi))\wedge\Phi) &=\star((X\lrcorner\theta^{\sharp}\lrcorner\Phi)\wedge\Phi)\\
&=-9(X^{\flat}\wedge\theta)_7-(X^{\flat}\wedge\theta)_{21}.
\end{align*}
Putting all together, we get
\begin{align*}
2\xi_X &=-(X\lrcorner\delta\Phi)_7-\frac{7}{24}(-3(X^{\flat}\wedge\theta)_7+(X^{\flat}\wedge\theta)_{21}-9(X^{\flat}\wedge\theta)_7-(X^{\flat}\wedge\theta)_{21})\\
&=-(X\lrcorner\delta\Phi)_7+\frac{7}{2}(X^{\flat}\wedge\theta)_7.
\end{align*}

Finally, by Lemma \ref{lem:Tdecmposition} and \eqref{eq:A2}, we have
\begin{align*}
X\lrcorner\delta\Phi &=X\lrcorner(\delta\Phi)_8+X\lrcorner(\delta\Phi)_{48}=-X\lrcorner\star(\theta\wedge\Phi)+X\lrcorner(\delta\Phi)_{48}\\
&=-X\lrcorner\theta^{\sharp}\lrcorner\Phi+X\lrcorner(\delta\Phi)_{48}\\
&=3(X^{\flat}\wedge\theta)_7-(X^{\flat}\wedge\theta)_{21}+X\lrcorner(\delta\Phi)_{48},
\end{align*}
which proves the last part.
\end{proof}

\begin{cor}
For each pure class $\mathcal{W}_8, \mathcal{W}_{48}$ the intrinsic torsion simplifes to
\begin{align*}
\mathcal{W}_{8} &:\, \xi_X=\frac{1}{4}(X^{\flat}\wedge\theta)_7,\\
\mathcal{W}_{48} &:\, \xi_X=-\frac{1}{2}(X\lrcorner\delta\Phi)_7.
\end{align*}
\end{cor}

\subsection{Intrinsic torsion in the literature}

The formula for intrinsic torsion has been established in an explicit form in \cite{CI}, \cite{SK2} and \cite{MM}. We do not focus on the approach from \cite{MM} since it uses the spinorial approach omitted here. We recall the first two approaches and outline the equivalence of these formulas with the one obtained in this note.

\subsubsection{First formula} Let us begin with the formula by Cabrera and Ivanov (Section 3 in \cite{CI}). It takes the form
\begin{equation*}
g(\xi_XY,Z)=\frac{1}{4}\bar{r}(\nabla\Phi)(X,Y,Z)
\end{equation*}
where $\bar{r}(\nabla\Phi)(X,Y,Z)=\frac{1}{8}\skal{X\lrcorner\nabla\Phi}{Y^{\flat}\wedge Z\lrcorner\Phi-Z^{\flat}\wedge Y\lrcorner\Phi}$ and $\skal{\cdot}{\cdot}$ is an inner product induced from the Riemannian metric $g$. It can be shown that this formula coincides with the formulas in Proposition \ref{prop:intrinsictorsion}. Let us provide a proof.

It was shown in \cite{CI} that, alternatively, we may write
\begin{equation*}
g(\xi_XY,Z)=\frac{1}{8}\skal{X\lrcorner d\Phi}{Y\wedge Z\lrcorner\Phi-Z\wedge Y\lrcorner\Phi}-\frac{7}{16}(X^{\flat}\wedge\theta)(Y,Z).
\end{equation*}
Since $\star(X\lrcorner d\Phi)=-X^{\flat}\wedge\delta\Phi$, thus
\begin{equation}\label{eq:CIequivalence}
\skal{X\lrcorner d\Phi}{Y\wedge Z\lrcorner\Phi-Z\wedge Y\lrcorner\Phi}=-\star(X^{\flat}\wedge\delta\Phi\wedge(Y\wedge Z\lrcorner\Phi-Z\wedge Y\lrcorner\Phi)).
\end{equation}
Moreover, for an $8$--form $\Omega$ we have
\begin{equation*}
\star(Y\lrcorner X\lrcorner\Omega)=(\star\Omega)X^{\flat}\wedge Y^{\flat}.
\end{equation*}
Replacing $\Omega$ by the right hand side of \eqref{eq:CIequivalence}, after long computations, we get
\begin{align*}
\xi_X &=\sum_{i,j}\skal{\xi_Xe_i}{e_i}e^{ij}\\
&=-\frac{1}{8}\star(\delta\Phi\wedge(X\lrcorner\Phi)-X\wedge(\delta\Phi\bullet\Phi)),
\end{align*}
where $\delta\Phi\bullet\Phi=\sum_i (e_i\lrcorner\delta\Phi)\wedge(e_i\lrcorner\Phi)$ is a $5$--form. By the formula for $\Phi$ we see that for any $3$--form $\gamma$
\begin{equation*}
\gamma\bullet\Phi=7\theta\wedge\Phi-\star\gamma.
\end{equation*}
Taking $\gamma=\delta\Phi$ we obtain the third formula for the intrinsic torsion from Proposition \ref{prop:intrinsictorsion}.

\subsubsection{Second formula} Now, let us describe the formula for intrinsic torsion derived from \cite{SK2}. Firstly, consider the natural action of the group $SO(8)$ on the $4$-form $\Phi_0$ in $\mathbb{R}^8$. Taking the infinitesimal action of the corresponding Lie algebra $\alg{so}(8)$ we get \cite{SK2}
\begin{equation*}
D:\, A\cdot \Phi_0 =\frac{d}{dt}({\rm exp}(tA)\cdot \Phi)_{t=0}=\sum_{i,m}a^m_i dx^i\wedge(\frac{\partial}{\partial x_m}\lrcorner\Phi),
\end{equation*}
where $A=(a^m_i)\in\alg{so}(8)$. We induce this action on $2$--tensors via $\Phi$ and the metric tensor $g$:
\begin{equation*}
D:\, A\cdot \Phi=\sum_{i,j,m}A_{ij}g^{jm} e^i\wedge(e_m\lrcorner\Phi),
\end{equation*}  
where $A=(A_{ij})$ is a tensor of type $(0,2)$ written in local orthonormal coordinates $(e_i)$. It can be shown that $D(\Lambda^2_7)=\Lambda^4_7$. Moreover, since $\nabla_X\Phi\in\Lambda^4_7$ for any vector field $X$, there is a $3$--tensor $\xi$ such that
\begin{equation*}
D(\xi_X)=\nabla_X\Phi.
\end{equation*}
Karigiannis calls $\xi$ the torsion. By the definition of intrinsic torsion and the fact that $\nabla^{{\rm Spin}(7)}\Phi=0$, it follows that $\xi$ is, in fact, intrinsic torsion.

\section{Non-existence of (second order) parallel locally conformal parallel ${\rm Spin}(7)$ structures}

In this section, we study ${\rm Spin}(7)$--structures of pure type $\mathcal{W}_8$ which are parallel in two ways described below:
\begin{itemize}
    \item {\it with parallel intrinsic torsion}: $\nabla^{{\rm Spin}(7)}\xi=0$,
    \item {\it second order nearly parallel}: $\nabla^{{\rm Spin}(7)}r=0$, 
\end{itemize}
where $\nabla^{{\rm Spin}(7)}$ is a minimal ${\rm Spin}(7)$--connection and $r$ is a symmetric $2$--tensor defined by
\begin{equation}\label{eq:tensor_r}
    r(X,Y) = \sum_i g(\xi_Xe_i,\xi_Ye_i).
\end{equation}

The second type of parallelism has been considered by the author \cite{KN2} in a general context of $G$--structures. It appears that this condition is quite natural and many geometric structures, such as the nearly K"ahler, nearly parallel $G_2$ or Kenmotsu manifolds, satisfy this condition.

We focus on locally conformal parallel ${\rm Spin}(7)$--structures, that is, we assume that the intrinsic torsion equals
\begin{equation*}
    \xi_XY = \frac{1}{4}(X^{\flat}\wedge\theta)_7.
\end{equation*}
This amounts to the fact that
\begin{equation}\label{eq:int_tor_w8}
    \xi_XY = \frac{1}{16}(g(X,Y)\theta - \theta(Y)X + \Phi(X,\theta,Y)).
\end{equation}

Let us begin with the first case, which is quite computational.

\begin{prop}
    A locally conformal parallel ${\rm Spin}(7)$--structure has parallel intrinsic torsion if and only if $\nabla^{{\rm Spin}(7)}_\theta=0$.
\end{prop}
\begin{proof}
Since $(\nabla^{{\rm Spin}(7)}_Z\xi)_XY = \nabla^{{\rm Spin}(7)}_Z(\xi_XY)-\xi_{\nabla^{{\rm Spin}(7)}_ZX}Y-\xi_X(\nabla^{{\rm Spin}(7)}_ZY)$, after some calculations and using the fact that $\nabla^{{\rm Spin}(7)}\Phi=0$, we get
\begin{equation*}
    (\nabla^{{\rm Spin}(7)}_Z\xi)_XY = \frac{1}{16}(g(X,Y)\nabla^{{\rm Spin}(7)}_Z\theta - (\nabla^{{\rm Spin}(7)}_Z\theta)(Y)X+\Phi(X,\nabla^{{\rm Spin}(7)}_Z\theta,Y)).
\end{equation*}
Clearly, if $\nabla^{{\rm Spin}(7)}\theta=0$, then $\nabla^{{\rm Spin}(7)}\xi=0$. Conversely, if $\nabla^{{\rm Spin}(7)}\xi=0$, then taking $X=Y$ we have
\begin{equation*}
    |X|^2\nabla^{{\rm Spin}(7)}_Z\theta - g(\nabla^{{\rm Spin}(7)}_Z\theta,X)X=0.
\end{equation*}
Choosing $X$ to be orthogonal to $\nabla^{{\rm Spin}(7)}_Z\theta$ we obtain $\nabla^{{\rm Spin}(7)}_Z\theta=0$ for any $Z$.
\end{proof}

By \cite[Proposition 3.3]{IPP} which states, among others, that on a compact locally conformal parallel ${\rm Spin}(7)$--structure the Lee form in necessary $\nabla$--parallel, we have the following immediate corollary.

\begin{cor}\label{cor:nonexistence}
    The compact locally conformal parallel ${\rm Spin}(7)$--structure cannot have parallel intrinsic torsion.
\end{cor}
\begin{proof}
    We have
    \begin{align*}
        (\nabla^{{\rm Spin}(7)}_Z\theta)X &= (\nabla_Z\theta)X -\theta(\xi_ZX)\\
        &=(\nabla_Z\theta)X-\frac{1}{16}(g(X,Z)|\theta|^2 - \theta(X)\theta(Z)).
    \end{align*}
    By assumptions $\nabla_Z\theta=0$ and $\nabla^{{\rm Spin}(7)}_Z\theta=0$. Thus $g(X,Z)|\theta|^2=\theta(X)\theta(Z)$ for any $X,Z$, which is impossible.
\end{proof}

 Let us move to the second notion of parallelism. 

\begin{thm}
A locally conformal parallel ${\rm Spin}(7)$--structure is second order nearly parallel if and only if $\nabla^{{\rm Spin}(7)}\theta=0$.
\end{thm}
\begin{proof}
Let us first compute $r$. We have
\begin{align*}
r(X,Y) &= \sum_ig(\xi_Xe_i,\xi_Ye_i) \\
&=\frac{1}{256}\sum_i g(g(X,e_i)\theta-\theta(e_i)X+\Phi(X,\theta,e_i),g(Y,e_i)\theta-\theta(e_i)Y+\Phi(Y,\theta,e_i))\\
&=\frac{1}{256}(2g(X,Y)\|\theta|^2-2\theta(X)\theta(Y)+\sum_i g(\Phi(X,\theta,e_i)\Phi(Y,\theta,e_i))).
\end{align*}
Let us simplify the last summand. Using Corollary 6.22 in \cite{SW} we have
\begin{equation*}
    \sum_i g(\Phi(X,\theta,e_i)\Phi(Y,\theta,e_i)) = -6(g(X,Y)|\theta|^2 + \theta(X)\theta(Y)).
\end{equation*}
Finally,
\begin{equation*}
    r(X,Y) = -\frac{1}{64}(|\theta|^2g(X,Y)-\theta(X)\theta(Y)).
\end{equation*}
Therefore,
\begin{align*}
    (\nabla^{{\rm Spin}(7)}_Zr)(X,Y) &=Z r(X,Y)-r(\nabla^{{\rm Spin}(7)}_ZX,Y)-r(X,\nabla^{{\rm Spin}(7)}_ZY)\\
    &=-\frac{1}{64}((Z|\theta|^2)g(X,Y)-(\nabla^{{\rm Spin}(7)}_Z\theta)(X)\theta(Y)-\theta(X)(\nabla^{{\rm Spin}(7)}_Z\theta)(Y)).
\end{align*}

Assume $\nabla^{{\rm Spin}(7)}r=0$. Substituting $X=Y$ orthogonal to $\theta$ we get that $|\theta|^2$ is a constant. Furthermore, taking $X=\theta$ and $Y$ orthogonal to $\theta$ and later $X=Y=\theta$ we arrive at $\nabla^{{\rm Spin}(7)}\theta=0$. Conversely, assume $\nabla^{{\rm Spin}(7)}\theta=0$. Then
\begin{equation*}
    (\nabla^{{\rm Spin}(7)}_Zr)(X,Y) = -\frac{1}{64}(Z|\theta|^2)g(X,Y).
\end{equation*}
Moreover, for any $Z$ we have
\begin{equation*}
    Z|\theta|^2 = 2g(\nabla^{{\rm Spin}(7)}_Z\theta,\theta) = 0.
\end{equation*}
This implies $\nabla^{{\rm Spin}(7)}r=0$.
\end{proof}

Analogously as Corollary \ref{cor:nonexistence} we have:

\begin{cor}
    There is no compact locally conformal parallel ${\rm Spin}(7)$--structure which is second order nearly parallel.
\end{cor}

\section{A new approach to well known divergence formulas}

We will use the divergence and integral formula obtained by the author in \cite{KN} for a general $G$--structure to derive well known formulas in the case of ${\rm Spin}(7)$--structures. Let $(M,g,\Phi)$ be a ${\rm Spin}(7)$--structure, $\xi$ its intrinsic torsion. We may decompose the Riemann curvature tensor $R$ with respect to the decomposition \eqref{eq:so8decomposition} as follows
\begin{equation*}
R(X,Y)=R(X,Y)_{\alg{spin}(7)}+R(X,Y)_{\alg{spin}(7)^{\bot}}.
\end{equation*}
Thus, we may define the $\alg{spin}(7)^{\bot}$--scalar curvature $s_{\alg{spin}(7)^{\bot}}=\sum_{i,j}g(R(e_i,e_j)_{\alg{spin}(7)^{\bot}}e_j,e_i)$. Moreover, we consider the scalar curvature $s^{{\rm Spin}(7)}$ of the curvature tensor $R^{{\rm Spin}(7)}$ of a connection $\nabla^{\rm Spin}$. We need three additional objects coming from the intrinsic torsion
\begin{equation*}
\xi^{\rm alt}_XY=\frac{1}{2}(\xi_XY-\xi_YX),\quad
\xi^{\rm sym}_XY=\frac{1}{2}(\xi_XY+\xi_YX),\quad
\chi=\sum_i \xi_{e_i}e_i.
\end{equation*}
We call $\chi$ the characteristic vector field. The divergence formulas taken from \cite{KN} in the case of a ${\rm Spin}(7)$--structure take the following forms
\begin{align}
2{\rm div}\,\chi &=s^{\rm Spin(7)}-s+|\chi|^2+|\xi^{\rm alt}|^2-|\xi^{\rm sym}|^2,\label{eq:divformula1}\\
{\rm div}\,\chi &=-\frac{1}{2}s_{\alg{spin}(7)^{\bot}}+|\chi|^2+|\xi^{\rm alt}|^2-|\xi^{\rm sym}|^2. \label{eq:divformula2}
\end{align}
Here $s$ is the scalar curvature of the Levi-Civita connection. Notice that
\begin{equation*}
|\xi^{\rm sym}|^2-|\xi^{\rm alt}|^2=\sum_{i,j}g(\xi_{e_i}e_j,\xi_{e_j}e_i).
\end{equation*}

\begin{rem}
Notice that above divergence formulas do not contain the following component
\begin{equation*}
\sum_{i,j}g([\xi_{e_i},\xi_{e_j}]_{{\rm spin}(7)^{\bot}}e_j,e_i),
\end{equation*}
which appears for general $G$--structure (replacing ${\rm Spin}(7)$ by $G$) \cite{KN}. This follows from the fact that in our situation it vanishes since $[\alg{spin}(7)^{\bot},\alg{spin}(7)^{\bot}]\subset\alg{spin}(7)$. 
\end{rem}

To derive above divergence formulas in the ${\rm Spin}(7)$ case, we will compute each component separately.  

\begin{lem}\label{lem:sspin7bot}
The scalar curvature $s_{\alg{spin}(7)^{\bot}}$ is equal to $\frac{1}{4}s$.
\end{lem}
\begin{proof}
By the definition of the projection $\pi_7$ we have
\begin{align*}
s_{\alg{spin}(7)^{\bot}} &=\frac{1}{4}g((R(e_i,e_j)+R(e_i,e_j)\lrcorner \Phi)e_j,e_i)\\
&=\frac{1}{4}\sum_{i,j}(g(R(e_i,e_j)e_j,e_i)+\frac{1}{4}\sum_{p,q}R_{ijpq}\varepsilon_{pqji},
\end{align*}
where $R_{ijpq}=g(R(e_i,e_j)e_p,e_q)$ and $\varepsilon_{pqji}=\Phi(e_p,e_q,e_j,e_i)$. By the Bianchi identity and a skew--symmetry of $\Phi$, second component vanishes. Thus, $s_{\alg{spin}(7)^{\bot}}=\frac{1}{4}s$.
\end{proof}

\begin{lem}\label{lem:characteristicvectorfield}
The characteristic vector field is equal to $\chi=\frac{7}{16}\theta^{\sharp}$.
\end{lem}
\begin{proof}
We identify $1$--forms with vector fields. Using \eqref{eq:A1} we get
\begin{align*}
8\chi &=\sum_i (-e_i\lrcorner e_i\lrcorner\delta\Phi-e_i\lrcorner\star((e_i\lrcorner\delta\Phi)\wedge\Phi)+\frac{7}{2}e_i\lrcorner(e^i\wedge\theta)+\frac{7}{2}e_i\lrcorner\star(e^i\wedge\theta\wedge\Phi))\\
&=\sum_i(-\star(e^i\wedge(e_i\lrcorner\delta\Phi)\wedge\Phi)+\frac{7}{2}(\theta-\theta(e_i)e_i))\\
&=-3\star(\delta\Phi\wedge\Phi)+\frac{49}{2}\theta\\
&=\frac{7}{2}\theta.
\end{align*}
\end{proof}

\begin{lem}\label{lem:symminusalt}
We have
\begin{equation*}
|\xi^{\rm sym}|^2-|\xi^{\rm alt}|^2=-\frac{35}{256}|\theta|^2+\frac{1}{16}|(\delta\Phi)_{48}|^2.
\end{equation*}
\end{lem}
\begin{proof}
Firstly, by Proposition \ref{prop:intrinsictorsion} we have
\begin{align*}
8\xi_{e_i}e_j &=-e_j\lrcorner e_i\lrcorner(\delta\Phi)_{48}-e_j\lrcorner\star(\Phi\wedge e_i\lrcorner(\delta\Phi)_{48}))\\
&+\frac{1}{2}\delta_{ij}\theta-\frac{1}{2}\theta(e_j)e_i+\frac{1}{2}e_j\lrcorner\star(\Phi\wedge e^i \wedge\theta).
\end{align*}
Thus
\begin{align*}
64(|\xi^{\rm sym}|^2-|\xi^{\rm alt}|^2) &=64\sum_{i,j}\skal{\xi_{e_i}e_j}{\xi_{e_j}e_i}\\
&=-6|(\delta\Phi)_{48}|^2+2\skal{e_j\lrcorner e_i\lrcorner(\delta\Phi)_{48}}{e_i\lrcorner\star(\Phi\wedge e_j\lrcorner(\delta\Phi)_{48}))}\\
&-\skal{e_j\lrcorner e_i\lrcorner(\delta\Phi)_{48}}{e_i\lrcorner\star(\Phi\wedge e^j \wedge\theta)}\\
&-\skal{e_j\lrcorner\star(\Phi\wedge e_i\lrcorner(\delta\Phi)_{48}))}{e_i\lrcorner\star(\Phi\wedge e_j\lrcorner(\delta\Phi)_{48}))}\\
&-e_i\lrcorner\star(\Phi\wedge e_i\lrcorner(\delta\Phi)_{48}))\theta\\
&-\skal{e_j\lrcorner\star(\Phi\wedge e_i\lrcorner(\delta\Phi)_{48}))}{e_i\lrcorner\star(\Phi\wedge e^j \wedge\theta)}\\
&+2|\theta|^2-\frac{1}{2}|\theta|^2+\frac{1}{4}|\theta|^2.
\end{align*}
The lemma follows by long computations using formula for $\Phi$, the fact that $(\delta\Phi)_{48}\wedge\Phi=0$ and properties \eqref{eq:A0}--\eqref{eq:A6}.
\end{proof}

Now we are ready to state and prove the main formulas, which are well known (see, for example, \cite{UF} and \\cite{SI})

\begin{thm}\label{thm:main}
On a ${\rm Spin}(7)$--structure $(M,g,\Phi)$ the following divergence formulas hold
\begin{align*}
7{\rm div}\theta^{\sharp} &=8(s^{\rm Spin(7)}-s)+\frac{21}{8}|\theta|^2-\frac{1}{2}|(\delta\Phi)_{48}|^2,\\
\frac{7}{2}{\rm div}\theta^{\sharp} &=-s+\frac{21}{8}|\theta|^2-\frac{1}{2}|(\delta\Phi)_{48}|^2.
\end{align*}
\end{thm} 
\begin{proof}
It suffices to apply the formula for the characteristic vector field of Lemma \ref{lem:characteristicvectorfield} and Lemmas \ref{lem:sspin7bot}, \ref{lem:symminusalt}.
\end{proof}

\begin{cor}\label{cor:scalarcurvature}
The scalar curvature of a ${\rm Spin}(7)$--structure induced by the Lee form $\theta$ and $\Lambda^3_{48}$ component of $\delta\Phi$ equals
\begin{equation}\label{eq:maindivrewritten}
s=\frac{21}{8}|\theta|^2-\frac{1}{2}|(\delta\Phi)_{48}|^2+\frac{7}{2}{\rm div}\,\theta.
\end{equation}
\end{cor}

We justify the divergence formula \eqref{eq:maindivrewritten} on certain examples. These examples has been already considered (see, for example, \cite{FMC}) but from other perspective.

\begin{exa}
Let $M(k)$ be a manifold defined as follows \cite{FMC}: choose a non--zero number $k\in\mathbb{R}$ and let $G(k)$ be the Lie group consisting of matrices of the form
\begin{equation*}
\left(\begin{array}{cccc}
e^{kz} & 0 & 0 & x \\
0 & e^{-kz} & 0 & y \\
0 & 0 & 1 & z \\
0 & 0 & 0 & 1
\end{array}\right),\quad x,y,z\in\mathbb{R}.
\end{equation*}
In other words, $G(k)$ is a group of affine transformations $\varphi:\mathbb{R}^3\to\mathbb{R}^3$ of the form $\varphi(x_0,y_0,z_0)=(e^{kz}x_0+x,e^{-kz}y_0+y,z_0+z)$. We may treat $(x,y,z)$ as a coordinate system on $G(k)$. Moreover, system $\{\eta_1=dx-kxdz, \eta_2=dy+kydy, \eta_3=dz\}$ forms a basis of right invariant one--forms. There is a discrete subgroup $\Gamma(k)$ such that the quotient manifold $M(k)=G(k)/\Gamma(k)$ is compact. Additionally, the considered system of one--forms on $G(k)$ descends to $M(k)$. Consider now a product manifold $M=G(k)\times\mathbb{T}^5$, where $\mathbb{T}^5$ is a $5$--dimensional torus. Then $M$ is a $8$-dimensional manifold. Let $\alpha_1,\ldots,\alpha_5$ be a basis of closed one--forms on $\mathbb{T}^5$. Put
\begin{equation*}
e^0=\alpha_1,\, e^1=\alpha_1,\, e^2=\alpha_3,\, e^3=\alpha_4,\, e^4=\alpha_5,\, e^5=\eta_1,\, e^6=\eta_2,\, e^7=\eta_3.
\end{equation*}
We define a metric tensor $g$ on $M$ such that frame of one--forms $e_0,e_1,\ldots,e_7$ is orthonormal. We have
\begin{equation*}
de^5=-ke^{57},\quad de^6=ke^{67},\quad \textrm{$de^i=0$ for remaining indices $i$}.
\end{equation*}
Define the $4$--form $\Phi$ by \eqref{eq:Phi} and consider a ${\rm Spin}(7)$--structure induced by $\Phi$. We have
\begin{equation*}
d\Phi=-ke^{01457}+ke^{02467}-ke^{23457}-ke^{13467}.
\end{equation*}
Thus,
\begin{equation*}
\delta\Phi=ke^{236}-ke^{135}+ke^{016}-ke^{025}.
\end{equation*}
This implies $\delta\Phi\wedge\Phi=0$. Since $d\Phi\neq 0$, it follows that defined ${\rm Spin}(7)$--structure is of type $\mathcal{W}_{48}$. We have
\begin{equation*}
|\delta\Phi|^2 = 4k^2.
\end{equation*}
Moreover, the intrinsic torsion equals
\begin{align*}
\xi_{e_0} &=\frac{k}{4}(e^{16}-e^{07}+e^{34}-e^{25}),\\
\xi_{e_1} &=\frac{k}{4}(e^{35}-e^{06}+e^{24}-e^{17}),\\
\xi_{e_2} &=\frac{k}{4}(e^{36}+e^{05}-e^{27}-e^{14}),\\
\xi_{e_6} &=\frac{k}{4}(e^{23}+e^{01}-e^{67}-e^{45})
\end{align*}
with the remaining components vanishing. Thus,
\begin{equation*}
\chi=\sum_i \xi_{e_i}e_i=0,\quad |\xi^{\rm sym}|^2-|\xi^{\rm alt}|^2=\frac{k^2}{4}=\frac{1}{16}|\delta\Phi|^2.
\end{equation*}
Let us compute the scalar curvature of $(M,g)$. Since $de^i(e_j,e_k)=-e_i([e_j,e_k])$, it follows that the nonzero Lie brackets $[e_j,e_k]$ are
\begin{equation*}
[e_5,e_7]=ke_5,\quad [e_7,e_5]=-ke_5,\quad [e_6,e_7]=-ke_6\,\quad [e_7,e_6]=ke_6.
\end{equation*}
Thus, using Koszul formula for the Levi-Civita connection, we get the following nonzero expressions
\begin{equation*}
\nabla_{e_5}e_7=ke_5,\quad \nabla_{e_6}e_7=-ke_6,\quad\nabla_{e_5}e_5=-ke_7,\quad \nabla_{e_6}e_6=ke_7.
\end{equation*}
This implies that the only nonzero components $R_{ijji}$ of the Riemann curvature tensor are
\begin{equation*}
R_{5775}=R_{7557}=R_{6776}=R_{7667}=-k^2,\quad R_{5665}=R_{6556}=k^2.
\end{equation*} 
Thus, the scalar curvature is $s=-2k^2$. Finally, the left-hand side of the divergence formula \eqref{eq:maindivrewritten} is $s=-2k^2$, which indeed coincides with $-\frac{1}{2}|\delta\Phi|^2$ as seen above. 
\end{exa}

\begin{exa}
Let $M=M(k)\times H/\Gamma\times\mathbb{T}^2$, where $M(k)$ and $\mathbb{T}$ are as before and $H/\Gamma$ is a quotient of a Heisenberg group $H$ by a certain discrete subgroup $\Gamma$, which are defined as follows: $H$ consists of matrices of the form
\begin{equation*}
\left(\begin{array}{ccc}
1 & x & y \\ 0 & 1 & z \\ 0 & 0 & 1
\end{array}\right),\quad x,y,z\in\mathbb{R},
\end{equation*}
whereas $\Gamma$ is a subgroup with $x,y,z\in\mathbb{Z}$. The $1$--orms $\mu_1=dx,\mu_2=dy,\mu_3=dz-xdy$ define a basis of left--invariant forms, which descent to $H/\Gamma$. Then $d\mu_3=-\mu_1\wedge\mu_2$. Denote by $\eta_1,\eta_2,\eta_3$ and by $\alpha_1, \alpha_2$ one--forms on $M(k)$ and $\mathbb{T}^2$ considered before. Put
\begin{equation*}
e^0=\alpha_1,\, e^1=\alpha_2,\, e^2=\mu_1,\, e^3=\mu_2,\, e^4=\mu_3,\, e^5=\eta_1,\, e^6=\eta_2,\, e^7=\eta_3.
\end{equation*}
Then
\begin{equation*}
de^4=-e^{23},\quad de^5=-ke^{57},\quad de^6=ke^{67}
\end{equation*}
with the remaining differentials being zero. Consider the Riemannian metric on $M$ such that $(e_i)$ forms an orthonormal basis.

Define $\Phi$ by \eqref{eq:Phi}. Then
\begin{equation*}
d\Phi=e^{01235}-e^{23567}-ke^{01457}+ke^{02467}-ke^{23457}-ke^{13467}
\end{equation*}
and
\begin{equation*}
\delta\Phi=-e^{467}+e^{014}+ke^{236}-ke^{135}+ke^{016}+ke^{025}.
\end{equation*}
This implies $7\theta=\star(\delta\Phi\wedge\Phi)=-2e_5$ and $d\Phi\neq\theta\wedge\Phi$. Hence, given ${\rm Spin}(7)$--structure is of type $\mathcal{W}_8\oplus\mathcal{W}_{48}$ and not of pure type. Moreover,
\begin{equation*}
|\delta\Phi|^2=4k^2+2\quad\textrm{and}\quad |\theta|^2=\frac{4}{49},\quad |(\delta\Phi)_{48}|^2=\frac{10}{7}+4k^2.
\end{equation*}
The intrinsic torsion equals
\begin{align*}
\xi_{e_2} &=-\frac{1}{8}(e^{25}+e^{07}-e^{34}-e^{16}),\\
\xi_{e_3} &=-\frac{1}{8}(e^{35}-e^{06}+e^{24}-e^{17}),\\
\xi_{e_4} &=-\frac{1}{8}(-e^{67}+e^{01}-e^{45}+e^{23}),\\
\xi_{e_5} &=-\frac{k}{4}(-e^{13}+e^{02}+e^{57}-e^{46},\\
\xi_{e_6} &=-\frac{k}{4}(e^{23}+e^{01}-e^{67}-e^{45})
\end{align*}
with the remaining elements vanishing. Thus
\begin{equation*}
\xi=\sum_i \xi_{e_i}e_i=-\frac{1}{8}e^5=\frac{7}{16}\theta
\end{equation*}
and
\begin{equation*}
|\xi^{\rm sym}|^2-|\xi^{\rm alt}|^2=\frac{5}{64}+\frac{k}{4}=-\frac{35}{256}|\theta|^2+\frac{1}{16}|(\delta\Phi)_8)|^2.
\end{equation*}
Let us derive the formula for the Levi-Civita connection. As before, the nonzero Lie brackets are
\begin{equation*}
[e_2,e_3]=-[e_3,e_2]=e_4,\quad [e_5,e_7]=-[e_7,e_5]=ke_5,\quad [e_6,e_7]=-[e_7,e_6]=-ke_7.
\end{equation*}
Therefore, nonzero components of the Levi-Civita connection, by the Koszul formula, are
\begin{align*}
&\nabla_{e_2}e_3=-\frac{1}{2}e_4,\quad \nabla_{e_2}e_4=\nabla_{e_4}e_2=-\frac{1}{2}e_3,\quad \nabla_{e_3}e_4=\nabla_{e_4}e_3=\frac{1}{2}e_2,\\ 
&\nabla_{e_5}e_7=ke_5,\quad \nabla_{e_6}e_7=-ke_6,\quad \nabla_{e_5}e_5=-ke_7,\quad \nabla_{e_6}e_6=ke_7.
\end{align*}
In particular,
\begin{equation*}
{\rm div}\theta=-\frac{2}{7}\sum_i g(\nabla_{e_i}e_5,e_i)=0.
\end{equation*}
Nonzero elements $R_{ijji}$ of the Riemann curvature tensor are
\begin{align*}
& R_{2332}=R_{3223}=-\frac{3}{4},\quad R_{2442}=R_{4224}=R_{3443}=R_{4334}=\frac{1}{4},\\
& R_{5775}=R_{7557}=R_{6776}=R_{7667}=-k^2,\quad R_{5665}=R_{6556}=k^2.
\end{align*}
This implies
\begin{equation*}
s=-\frac{1}{2}-2k^2.
\end{equation*}
Hence the left hand side of \eqref{eq:maindivrewritten} is $-\frac{1}{2}-2k^2$, which is indeed equal to the right hand side
\begin{equation*}
\frac{21}{8}\cdot \frac{4}{49}-\frac{1}{2}(\frac{10}{7}+4k^2)=-2k^2-\frac{1}{2}.
\end{equation*}
\end{exa}

\end{document}